\documentclass[11pt]{amsart}
\usepackage{amsmath,amssymb}

\newtheorem{thm}{Theorem}[section]

\newtheorem{prop}[thm]{Proposition}
\newtheorem{lem}[thm]{Lemma}

\theoremstyle{definition}
\newtheorem{exam}[thm]{Example}

\newtheorem{defi}[thm]{Definition}
\newtheorem{rem}[thm]{Remark}
\newtheorem{conj}[thm]{Conjecture}

\begin{document}

\author{Morimichi Kawasaki}

\address[Morimichi Kawasaki]{Graduate School of Mathematical Sciences, the University of Tokyo, 3-8-1 Komaba, Meguro-ku, Tokyo 153-0041, Japan and Center for Geometry and Physics, Institute for Basic Science (IBS), Pohang 37673, Republic of Korea}
\email{kawasaki@ibs.re.kr}
\title{Heavy subsets and non-contractible trajectories}
\maketitle
\begin{abstract}
Biran, Polterovich and Salamon defined a relative symplectic capacity which indicates the existence of 1-periodic non-contractible closed trajectories of Hamiltonian isotopies.
Many of researches have used the Hamiltonian Floer theory on non-contractible trajectories for giving upper bounds of Biran-Polterovich-Salamon's capacities.
However, in the present paper, we use the Oh-Schwarz spectral invariants which are defined in terms of the Hamiltonian Floer theory on contractible trajectories for a similar purpose.
\end{abstract}

\section{Introduction}
A subset $X$ of a symplectic manifold $M$ is said to be \textit{displaceable} if $X$ is displaceable by some Hamiltonian diffeomorphism (we give a more precise definition in Section \ref{Preliminaries}).
A subset $X$ is said to be \textit{non-displaceable} otherwise.
Relationship between (non-)displaceability and the existence problem of (non-trivial, non-contractible) periodic trajectories on Hamiltonian dynamics is one of interesting topics in symplectic topology.

As the author knows, the most classical theorem in this topic is the energy capacity inequality.
Hofer and Zehnder defined a symplectic capacity called the Hofer-Zehnder capacity which indicates the existence of non-trivial periodic orbits on autonomous Hamiltonian dynamics.
The energy capacity inequality states that the Hofer-Zehnder capacity has an upper bound by displacement energy.
The energy capacity inequality on the Euclidean space is proved by Hofer and Zehnder (\cite{HZ}) and some researchers generalized their work to more general symplectic manifolds (for example, see \cite{Sc}, \cite{U}).

Biran, Polterovich and Salamon defined a relative symplectic capacity which indicates the existence of non-contractible periodic trajectories of Hamiltonian isotopies.
In the present paper, we consider relationships between Biran-Polterovich-Salamon's capacity and (non-)displaceability.

Now, we give the precise definition of Biran-Polterovich-Salamon's capacity.
For a compact subset $Y$ of an open symplectic manifold  $(N,\omega)$ and a free homotopy class $\alpha\in [S^1,N]$,
Biran, Polterovich and Salamon \cite{BPS} defined a relative symplectic capacity $C_{BPS}(N,Y;\alpha)$ by
\[C_{BPS}(N,Y;\alpha)=\inf \{K>0 ; \forall H\in \mathcal{H}_K(N,Y), \mathcal{P}(H;\alpha)\neq \emptyset \},\]
where 
\[\mathcal{H}_K(N,Y)=\{ H\in C_c^\infty(S^1\times N) ; \inf_{S^1\times Y} H \geq K \},\]
and $\mathcal{P}(H;\alpha)$ is the set of 1-periodic trajectories of the Hamiltonian isotopy generated by the Hamiltonian function $H$ in the class $\alpha$.

Biran, Polterovich and Salamon proved the following theorem by showing non-vanishing of the homomorphism from a symplectic homology to a relative symplectic homology.

\begin{thm}[\cite{BPS}]\label{BPS1}
Let $N$ be a connected closed Riemannian manifold and $\alpha\in[S^1,N]$ a non-trivial homotopy class of free loops in $N$.
Assume that $N$ is the $n$-dimensional torus or has the Riemannian metric whose sectional curvature is negative.
Then
\[C_{BPS}(B^\ast N,N;\alpha)=l_\alpha,\]
 where $l_\alpha$ is the infimum of length of closed geodesics in the class $\alpha$.
Here let $(B^\ast N,\omega_N)$ denote the unit ball subbundle of the cotangent bundle with the standard symplectic form $\omega_N$ and let $N$ denote the zero section of $B^\ast N$.
\end{thm}

After the above work by Biran, Polterovich and Salamon, Weber \cite{W} proved that Theorem \ref{BPS1}  holds for any connected closed Riemannian manifold $N$ and
Niche \cite{N} gave upper bounds of Biran-Polterovich-Salamon's capacities for twisted cotangent bundles.

One of reasons why $C_{BPS}(B^\ast N,N;\alpha)$ is finite in their cases is that the zero-section $N$ is non-displaceable in $B^\ast N$.
Indeed, Biran, Polterovich and Salamon essentially proved the following proposition.

\begin{prop}[Proposition 3.3.2 of \cite{BPS}]\label{BPS displaceability}
Let $(N,\omega)$ be a connected open symplectic manifold and $Y$ a compact subset of $N$. Let $\alpha$ be a non-trivial homotopy class of free loops.
Assume that there exists a Hamiltonian function $H\colon S^1\times N\to\mathbb{R}$ with compact support such that $Y\cap\phi^1_H(Y)=\emptyset$ and $\mathcal{P}(H;\alpha)=\emptyset$.
Then $C_{BPS}(N,Y;\alpha)=\infty$.
Here $\{\phi_H^t\}$ is the Hamiltonian isotopy generated by $H$.
\end{prop}

Thus, we would like to know the problem whether Biran-Polterovich-Salamon's capacity is finite or not on non-displaceable subsets in general.

In the present paper, we consider Biran-Polterovich-Salamon's capacity in a special case and we prepare some notions now.

For  $R=(R_1,\ldots,R_n)\in (\mathbb{R}_{>0})^n $, let $I_R^n$ be the open subset of $\mathbb{R}^n$ defined by $I_R^n=\{ p=(p_1,\ldots,p_n)\in\mathbb{R}^n;|p_i|<R_i\text{ for }i=1,\ldots,n\}$.
We consider the standard symplectic form $\omega_0=dp_1\wedge dq_1+\cdots +dp_n\wedge dq_n$ on $I_R^n\times T^{n}$ with coordinates $(p,q)=(p_1,\ldots,p_n,q_1,\ldots,q_n)$, where we regard $T^n$ as $(\mathbb{R}/\mathbb{Z})^n$.
We denote the zero-section $\{0\}\times T^n$ of $I_R^n\times T^n$ by $T^n$.

Let $(M,\omega)$ be a connected symplectic manifold and $X$ a compact subset of $M$. 
For $e=(e_1,\ldots,e_n)\in \mathbb{Z}^n$ and $R=(R_1,\ldots,R_n)\in (\mathbb{R}_{>0})^n$, we define the relative symplectic capacity $C(M,X,R;e)$ by
\[C(M,X,R;e)=C_{BPS}(M\times I_R^n\times T^n, X\times T^n;(0_M,e)).\]
Here, we fix the symplectic form $\operatorname{pr}_1^\ast\omega +\operatorname{pr}_2^\ast\omega_0$ on $ M\times I_R^n\times T^n$.

We pose the following conjecture.
For a positive integer $n$, a subset $X$ of a symplectic manifold $M$ is \textit{$n$-stably displaceable} if $X\times T^n$ is displaceable in $M\times T^\ast T^n$.
A subset $X$ is \textit{$n$-stably non-displaceable} otherwise.
$1$-stably displaceable subsets are called stably displaceable (\cite{EP09}).
Note  that any displaceable subset is $n$-stably displaceable for any $n$.
\begin{conj}\label{sn conj}
Let $X$ be an $n$-stably non-displaceable compact subset of a closed symplectic manifold $(M,\omega)$.
Then the equality
\[C(M,X,R;e)=\sum_{i=1}^nR_i\cdot |e_i|\]
holds for any elements $e=(e_1,\ldots,e_n)$ and $R=(R_1,\ldots,R_n)$ of $\mathbb{Z}^n$ and $(\mathbb{R}_{>0})^n$, respectively.
\end{conj}

In Section \ref{conj section}, we give an example such that $C(M,X,R;e)=+\infty$ even though $X$ is non-displaceable.


In Section \ref{measure observation}, we introduce a relative symplectic capacity $C^P$ which is defined in terms of invariant measures of (time-independent) Hamiltonian flow and satisfies $C^P(M,X,R;e)\leq C(M,X,R;e)$.
We give the following theorem which supports Conjecture \ref{sn conj}.
\begin{thm}\label{measure capacity ineq}
Let $(M,\omega)$ be a closed symplectic manifold and $X$ an $n$-stably non-displaceable compact subset of $M$.
Then
\[C^P(M,X,R;e)= \sum_{i=1}^nR_i\cdot |e_i|,\]
for any elements  $e=(e_1,\ldots,e_n)$ and $R=(R_1,\ldots,R_n)$ of $\mathbb{Z}^n$ and $(\mathbb{R}_{>0})^n$, respectively.
\end{thm}
Theorem \ref{measure capacity ineq} is obtained as a corollary of Polterovich's theorem in \cite{P14} (see Section \ref{measure observation}).


However, our main theorem is another one which also supports Conjecture \ref{sn conj}.
To explain our main theorem, we prepare some notions.

For a real number $\lambda$, a symplectic manifold $(M,\omega)$ is called $\lambda$-\textit{monotone} if  $[\omega]=\lambda c_1$ on $\pi_2(M)$ and \textit{monotone} if $(M,\omega)$ is $\lambda$-monotone for some non-negative $\lambda$.
Here $c_1$ is the first Chern class of $TM$ with respect to an almost complex structure compatible with $\omega$.

Entov and Polterovich (\cite{EP09}) defined heaviness for closed subsets of symplectic manifolds in terms of the Hamiltonian Floer theory on contractible trajectories (see Section \ref{spectral invariants and heavy subsets}).
Heavy subsets are known to be $n$-stably non-displaceable for any $n$ and thus non-displaceable.

Our main theorem is the following one.
\begin{thm}\label{capacity statement}
Let $X$ be a heavy subset of a $2m$-dimensional connected closed $\lambda$-monotone symplectic manifold $(M,\omega)$.
Then $C(M,X,R;e)\leq2\sum_{i=1}^nR_i\cdot |e_i|+\max\{0,\lambda(m+n)\}$ for any elements $e=(e_1,\ldots,e_n)$ and $R=(R_1,\ldots,R_n)$ of $\mathbb{Z}^n$ and $(\mathbb{R}_{>0})^n$, respectively.
\end{thm}

We can rewrite Theorem \ref{capacity statement} in the following form.

\begin{thm}\label{new characterization of heaviness}
Let $X$ be a heavy subset of a $2m$-dimensional connected closed $\lambda$-monotone symplectic manifold $(M,\omega)$.
Let $e=(e_1,\ldots,e_n)$ and $R=(R_1,\ldots,R_n)$ be elements of $\mathbb{Z}^n$ and $(\mathbb{R}_{>0})^n$, respectively.
We fix the symplectic form $\operatorname{pr}_1^\ast\omega+\operatorname{pr}_2^\ast\omega_0$ on $ M\times I_R^n\times T^n$, where $\operatorname{pr}_1\colon  M\times I_R^n\times T^n\to M$ and $\operatorname{pr}_2\colon  M\times I_R^n\times T^n \to I_R^n\times T^n$ are the the projections defined by $\operatorname{pr}_1(x,p,q)=x$ and $\operatorname{pr}_2(x,p,q)=(p,q)$.
Let $F\colon S^1\times M\times I_R^n\times T^n\to \mathbb{R}$ be a Hamiltonian function with compact support such that 
\[F|_{S^1\times X\times T^n}\geq 2\sum_{i=1}^nR_i\cdot |e_i|+\max\{0,-\lambda(m+n)\}.\]
Then the Hamiltonian isotopy $\{\phi_F^{t}\}_{t\in\mathbb{R}}$ has a 1-periodic trajectory in the free loop homotopy class $(0_M,e)\in[S^1,M\times I_R^n\times T^n]$. 
\end{thm}

Many of other works have used the Hamiltonian Floer theory on non-contractible trajectories to give upper bounds of BPS capacities (\cite{BPS}, \cite{W}, \cite{N}, \cite{X}).
Other work (for example, see \cite{G}, \cite{GG}) also uses such Floer theory to find non-contractible trajectories.
However, in the present paper, we use the Hamiltonian Floer theory on contractible trajectories to give an upper bound of BPS capacities.
More precisely, we use the Oh-Schwarz spectral invariants (see Section \ref{spectral invariants and heavy subsets}) which are defined in terms of the Hamiltonian Floer theory on contractible trajectories.

In contrast, for a displaceable compact subset $X$, we have the following result.
\begin{prop}\label{our displaceability}
Let $(M,\omega)$ be a connected symplectic manifold and $X$ a displaceable compact subset of $M$.
Let $e=(e_1,\ldots,e_n)$ and $R=(R_1,\ldots,R_n)$ be elements of $\mathbb{Z}^n$ and $(\mathbb{R}_{>0})^n$ such that $R_k\cdot |e_k|>E(X)$ for some $k$, respectively.
Here $E(X)$ denotes the displacement energy of $X$ (see Section \ref{displaceable}).
Then $C(M,X,R;e)=\infty$.
\end{prop}

\begin{exam}\label{torus and cp^n}
Let $(\mathbb{C}P^m,\omega_{FS})$ be the complex projective space with the Fubini-Study form. 
Let $\Phi\colon\mathbb{C}P^m\to\mathbb{R}^m$ be the moment map defined by 
\[\Phi([z_0:{\ldots}:z_m])=(\frac{|z_0|^2}{|z_0|^2+\cdots+|z_m|^2},\ldots,\frac{|z_m|^2}{|z_0|^2+\cdots+|z_m|^2}).\]
The Clifford torus $\Phi^{-1}(y_0)$ is a heavy subset of $(\mathbb{C}P^m,\omega_{FS})$ where $y_0=(\frac{1}{m+1},\ldots,\frac{1}{m+1})$ and $(\mathbb{C}P^m,\omega_{FS})$ is a monotone symplectic manifold.
Thus Theorem \ref{capacity statement} implies the finiteness of $C(\mathbb{C}P^m,\Phi^{-1}(y_0),R;e)$ for any elements $e=(e_1,\ldots,e_m)$ and $R=(R_1,\ldots,R_m)$ of $\mathbb{Z}^m$ and $(\mathbb{R}_{>0})^m$, respectively.

Lemma 5.1 of \cite{BEP} essentially proves that there exists a positive constant $P$ such that $E(\Phi^{-1}(y))<P$ for any element $y\neq y_0$ of $\mathbb{R}^m$.
Thus for any element $y\neq y_0$ of $\mathbb{R}^m$, Proposition \ref{our displaceability} implies $C(\mathbb{C}P^m,\Phi^{-1}(y),R;e)=\infty$ for any elements $e$ and $R$ of $\mathbb{Z}^m$ and $(\mathbb{R}_{>0})^m$  such that $R_k\cdot |e_k|>P$ for some $k$, respectively.
\end{exam}

The present paper is organized as follows.
We review some definitions in symplectic topology in Section \ref{Preliminaries} and spectral invariants in Section \ref{spectral invariants and heavy subsets} which are needed to prove Theorem \ref{new characterization of heaviness} in Section \ref{prf of main theorem}.
We discuss in Section \ref{non-monotone} the existence of periodic trajectories of period not more than 1.
In Section \ref{displaceable}, we look at the capacity of displaceable subsets and prove Proposition \ref{our displaceability}.
In Section \ref{compressible section}, we discuss generalizations of our main Example \ref{torus and cp^n}.
In Sections \ref{conj section}, we give a counter example when the assumption that $X$ is $n$-stably non-displaceable in Conjecture \ref{sn conj} is replaced by that $X$ is non-displaceable.
In Section \ref{measure observation}, we define the relative capacity $C^P$ and prove Theorem \ref{measure capacity ineq}.


\subsection*{Acknowledgement}
The author would like to thank his advisor Professor Takashi Tsuboi for his helpful advices.
He also thanks Professor Leonid Polterovich, Daniel Rosen and Egor Shelukhin for the faithful discussion.
Especially, Leonid encouraged him when he found a critical mistake in the previous version of this paper.
The author also thanks Hiroyuki Ishiguro, Professor Kaoru Ono and Ryuma Orita for some comments.
He is supported by  IBS-R003-D1, the Grant-in-Aid for Scientific Research (KAKENHI No. 25-6631) and the Grant-in-Aid for JSPS fellows. This work was 
supported by the Program for Leading Graduate Schools, MEXT, Japan.

\section{Preliminaries}\label{Preliminaries}

For a Hamiltonian function $H\colon M\to\mathbb{R}$ with compact support, we define the \textit{Hamiltonian vector field} $X_H$ associated with $H$ by
\[\omega(X_H,V)=-dH(V)\text{ for any }V \in \mathcal{X}(M),\]
where $\mathcal{X}(M)$ denotes the set of smooth vector fields on $M$.

Let $S^1$ denote $\mathbb{R}/\mathbb{Z}$.
For a (time-dependent) Hamiltonian function $H\colon S^1\times M\to\mathbb{R}$ with compact support and for $t \in S^1$, we define $H_t\colon M\to\mathbb{R}$ by $H_t(x)=H(t,x)$. 
We denote the Hamiltonian vector field associated with $H_t$ by $X_H^t$ and denote by $\{\phi_H^t\}_{t\in\mathbb{R}}$ the isotopy generated by $X_H^t$ such that $\phi^0=\mathrm{id}$.
For $x\in M$, we denote by $\gamma_H^x{\colon}[0,1]\to{M}$ the path defined by $\gamma_H^x(t)=\phi_H^t(x)$.

$\phi_H^1$ is called \textit{the Hamiltonian diffeomorphism generated by the Hamiltonian function $H$}.
For a symplectic manifold $(M,{\omega})$, let $\mathrm{Ham}(M,\omega)$ denote the group of Hamiltonian diffeomorphisms of $(M,{\omega})$.

A subset $X$ of $M$ is said to be \textit{displaceable} if $\bar{X}\cap\phi_H^1(X)=\emptyset$ for some Hamiltonian function $H\colon S^1\times M\to\mathbb{R}$, where $\bar{X}$ is the topological closure of $X$.
A subset $X$ is said to be \textit{non-displaceable} otherwise.

We denote the free loop space $C^\infty (S^1,M)$ of $M$ by $\mathcal{L}M$.
For $z\in \mathcal{L}M$, we denote its free homotopy class by $[z]\in [S^1,M]$.
Let $\operatorname{ev}\colon \mathcal{L}M\to M$ be the evaluation map defined by $\operatorname{ev}(z)=z(0)$.
For a given class $\alpha\in [S^1,M]$, we define the subset $\mathcal{L}_\alpha M$ of $\mathcal{L}M$ by $\mathcal{L}_\alpha M=\{ z\in \mathcal{L}M ; [z]= \alpha \}$.
For a Hamiltonian function $H\colon S^1\times M\to\mathbb{R}$, we define the set of 1-periodic trajectories of $\{\phi_H^t\}_{t\in\mathbb{R}}$ in the class $\alpha$ by
\[\mathcal{P}(H;\alpha)=\{ z\in \mathcal{L}_\alpha M ; \dot{z}(t)=X_H^t(z(t)) \}.\]
We define the covering space $\tilde{\mathcal{L}}_{0_M}(M)$ of $\mathcal{L}_{0_M}(M)$ by
\[\tilde{\mathcal{L}}_{0_M}(M)=\{ u \in C^\infty(D^2,M) ;  u|_{\partial D^2}\in\mathcal{L}_{0_M}(M) \}/\sim.\]
Here $u\sim u^\prime$ if $u|_{\partial D^2}=u^\prime|_{\partial D^2}$, $\omega(\bar{u}\sharp u^\prime)=0$ and $c_1(\bar{u}\sharp u^\prime)=0$, where $\sharp$ denotes the map from the sphere obtained from $u$ with the reversed orientation and $u^\prime$ by gluing along their common boundary.
We also define the covering space $\tilde{\mathcal{P}}(H)$ of $\mathcal{P}(H;0_M)$ by
\[\tilde{\mathcal{P}}(H)=\{ [z,u]\in \mathcal{P}(H;0_M)\times C^\infty(D^2,M) ;  u|_{\partial D^2}=z\}/\sim.\]
Here $[z,u]\sim[z^\prime,u^\prime]$ if $z=z^\prime$, $\omega(\bar{u}\sharp u^\prime)=0$ and $c_1(\bar{u}\sharp u^\prime)=0$.

\section{Spectral invariants and heavy subsets}\label{spectral invariants and heavy subsets}

\subsection{Spectral invariants}\label{introduction to spectral invariant}

For a $2m$-dimensional closed connected symplectic manifold $(M,\omega)$, we define
\[ {\Gamma}=\frac{\pi_2(M)}{\operatorname{Ker}(c_1)\cap\operatorname{Ker}([\omega])}. \]
The Novikov ring $\Lambda$ of the closed symplectic manifold $(M,\omega)$ is defined as follows:
\[ {\Lambda}=\left\{\sum_{A\in\Gamma}a_AA;a_A\in\mathbb{Z}_2,\forall R\in\mathbb{R},\#\{A;a_A\neq{0},\int_A{\omega}<R\}<\infty\right\}.\]
The quantum homology $QH_\ast(M,\omega)$ is a $\Lambda$-module which is isomorphic to $H_\ast(M;\mathbb{Z}_2)\otimes_{\mathbb{Z}_2}\Lambda$ and has a ring structure with the multiplication called the \textit{quantum product}.

For a Hamiltonian function $H\colon S^1\times M\to\mathbb{R}$, the action functional $\mathcal{A}_H\colon \tilde{\mathcal{L}}_{0_M}M\to\mathbb{R}$ is given by
\[\mathcal{A}_H([z,u])=\int_0^1H(t,z(t))dt-\int_{D^2}u^\ast\omega.\]
Then we regard $\tilde{\mathcal{P}}(H)$ as the set of critical points of $\mathcal{A}_H$.

We define the \textit{non-degeneracy} of Hamiltonian functions as follows:
\begin{defi}
A Hamiltonian function $H\colon S^1\times M\to\mathbb{R}$ is called \textit{non-degenerate} if for any element  $z$ of $\mathcal{P}(H;0_M)$, 1 is not an eigenvalue of the differential $(d\phi_H^1)_{z(0)}$.
\end{defi}

When $H$ is non-degenerate, the Floer chain complex $CF_\ast(H)$ is generated by $\tilde{\mathcal{P}}(H)$ as a module over $\mathbb{Z}_2$.
Since there exists a natural action of $\Lambda$ on $CF_\ast(H)$, we regard $CF_\ast(H)$ as a module over $\Lambda$.
The complex $CF_\ast(H)$ is graded by the Conley-Zehnder index $\operatorname{ind}_{CZ}$ (\cite{SZ}).
Note that $\operatorname{ind}_{CZ}([z,u\sharp A])=\operatorname{ind}_{CZ}([z,u])-2c_1(A)$ for any map $A\in \pi_2(M)$ in our convention.
Let $F\colon M\to\mathbb{R}$ be a Morse function on $M$ and $x$ a critical point of $F$.
Assume that $dF$ is $C^1$-small near $x$.
Then $\operatorname{ind}_{Morse}(x)=m-\operatorname{ind}_{CZ}([x,c_x])$,
where $c_x$ is a trivial capping disk and $\operatorname{ind}_{Morse}$ is the Morse index.
We formally obtain the boundary map of this complex by counting isolated negative gradient flow lines of $\mathcal{A}_H$ and we define its homology group $HF_\ast(H)$ which is called \textit{the Hamiltonian Floer homology on contractible trajectories} of $H$.

There exists a natural isomorphism $\Phi\colon QH_\ast(M,\omega)\to HF_\ast(H)$.
We call this isomorphism the PSS isomorphism (\cite{PSS}).

Given an element $A=\sum_i a_i[z_i,u_i]$ of $CF_\ast(H)$, we define the action level $l_H(A)$ of $A$ by
\[l_H(A)=\max\{\mathcal{A}_H([z_i,u_i]);a_i\neq 0\}.\]

For a non-zero element $a$ of $QH_\ast(M,\omega)$, we define the spectral invariant associated to $H$ and $a$ by
\[c(a,H)=\inf\{l_H(A);[A]=\Phi(a)\}.\]

The following proposition summarizes the properties of spectral invariants which we need to show our result.

\begin{prop}[\cite{Oh}]\label{Oh}
The spectral invariant has the following properties.

\begin{description}
\item[(1)\textit{Lipschitz property}]The map $H{\mapsto}c(a,H)$ is Lipschitz on $C^\infty(S^1\times M)$ with respect to the $C^0$-norm,
\item[(2)\textit{Homotopy invariance}] Assume that Hamiltonian functions $F,G\colon$ $S^1\times M\to\mathbb{R}$ are normalized \textit{i.e.} $\int_M F_t(x)\omega^m=0, \int_M G_t(x)\omega^m=0$ for any $t\in S^1$ and satisfy $\phi^1_F=\phi^1_G$ and that their Hamiltonian isotopies $\{\phi^t_F\}$ and $\{\phi^t_G\}$ are homotopic relative to endpoints. Then $c(a,F)=c(a,G)$,

\item[(3)\textit{Triangle inequality}]$c(a\ast{b},F\sharp G) \leq c(a,F)+c(b,G)$ for any Hamiltonian functions $F,G\colon S^1\times M\to\mathbb{R}$, where $\ast$ denotes the quantum product.
Here the Hamiltonian function $F{\sharp}G\colon S^1\times M\to\mathbb{R}$ is defined by $(F{\sharp}G)(t,x)=F(t,x)+G(t,(\phi_F^t)^{-1}(x))$ whose Hamiltonian isotopy is $\{\phi_F^t\phi_G^t\}$.
\end{description}
\end{prop}

For a general Hamiltonian function $H\colon S^1\times M\to\mathbb{R}$ ($H$ can be degenerate), we define the spectral invariant $c(a,H)$ by the Lipschitz property.
Then the spectral invariant defined for general Hamiltonian functions also satisfies the properties in Proposition \ref{Oh}.

\subsection{Heaviness}

A series of Entov and Polterovich's work (\cite{EP03}, \cite{BEP}, \cite{EP06} and \cite{EP09}) gave applications of the Oh-Schwarz spectral invariants to non-displaceability problem.
In \cite{EP09}, they defined the notion of \textit{heaviness} of closed subsets in closed symplectic manifolds.

For an idempotent $a$ of the quantum homology $QH_\ast(M,\omega)$, we define the functional  $\zeta_a{\colon}$ $C^\infty(S^1\times M)\to\mathbb{R}$ as the stabilization of $c(a,\cdot);$
\[\zeta_a(H)=\lim_{l\to\infty}\frac{c(a,H^{\natural l})}{l},\]
where $H^{\natural l}\colon S^1\times M\to\mathbb{R}$ is defined by $H^{\natural l}(t,x)=lH(lt,x)$.

\begin{defi}[\cite{EP09}]\label{definition of heavy}
Let $(M,\omega)$ be a closed symplectic manifold and $a$ an idempotent of the quantum homology $QH_\ast(M,\omega)$.
A closed subset $X$ of $(M,\omega)$ is said to be $a$-\textit{heavy} if
\[\zeta_a(H)\geq\inf_{S^1\times X}H,\]
for any (time-dependent) Hamiltonian function $H\colon S^1\times M\to\mathbb{R}$.
A closed subset $X$ of $(M,\omega)$ is called \textit{heavy} if $X$ is $a$-heavy for some idempotent $a$ of $QH_\ast (M,\omega)$.
\end{defi}

Entov and Polterovich \cite{EP09} proved that every heavy subset is non-displaceable (\cite{EP09} Theorem 1.4).

\begin{rem}
The above definition of heaviness is different from the one of \cite{EP09} (in their definition, they considered only autonomous Hamiltonian functions).
However, as noted in \cite{S14}, the above definition is known to be equivalent to the one of \cite{EP09}.
\end{rem}

\begin{exam}\label{torus}
On the torus $T^n_R\times T^n=\mathbb{R}/2R_1\mathbb{Z}\times \cdots\times \mathbb{R}/2R_n\mathbb{Z}\times( \mathbb{R}/\mathbb{Z})^n$ with coordinates $(p,q)=(p_1,\ldots,p_n,q_1,\ldots,q_n)$,
we fix the standard symplectic form $\omega_0=dp_1\wedge dq_1+\cdots +dp_n\wedge dq_n$.
Entov and Polterovich \cite{EP09} proved that for any element $R=(R_1,\ldots,R_n)$ of $(\mathbb{R}_{>0})^n $, $\{0\}\times T^n$ is a heavy subset of $T^n_R\times T^n$.
\end{exam}

\section{Proof of Theorem \ref{new characterization of heaviness} }\label{prf of main theorem}
To prove Theorem \ref{new characterization of heaviness}, we give an upper bound of the spectral invariant associated to a Hamiltonian function $F\colon S^1\times M\times I_{R(2\epsilon)}^n\times T^n\to\mathbb{R}$ such that  $\mathcal{P}(F;(0_M,e))=\emptyset$.
Here, for $R=(R_1,\ldots,R_n)\in (\mathbb{R}_{>0})^n$ and a positive real number $\epsilon$ with  $\epsilon <\min\{R_1,\ldots,R_n\}$, let $R(\epsilon)$ denote $(R_1-\epsilon,\ldots,R_n-\epsilon)\in (\mathbb{R}_{>0})^n$.

\begin{prop}\label{Irie and Seyfaddini}
Let $(M,\omega)$ be a $2m$-dimensional connected closed $\lambda$-monotone symplectic manifold.
Let $e=(e_1,\ldots,e_n)$ and $R=(R_1,\ldots,R_n)$ be elements of $\mathbb{Z}^n$ and $(\mathbb{R}_{>0})^n$, respectively.
For a positive real number $\epsilon$ with $2\epsilon <\min\{R_1,\ldots,R_n\}$, let $U_\epsilon$ be the open subset of $T^n_R\times T^n$ defined by
\[U_\epsilon=\{(p,q)\in T^n_R\times T^n; p\in I_{R(2\epsilon)}\}.\]
We fix the symplectic form $\operatorname{pr}_1^\ast\omega +\operatorname{pr}_2^\ast\omega_0$ on $ M\times T_R^n\times T^n$, where $\operatorname{pr}_1\colon  M\times T_R^n\times T^n\to M$ and $\operatorname{pr}_2\colon  M\times T_R^n\times T^n \to T_R^n\times T^n$ are the projections defined by $\operatorname{pr}_1(x,p,q)=x$ and $\operatorname{pr}_2(x,p,q)=(p,q)$.
Then for any Hamiltonian function $F\colon S^1\times M\times U_\epsilon\to\mathbb{R}$ with compact support such that $\mathcal{P}(F;(0_M,e))=\emptyset$,
\[c([M\times T^n_R\times T^n],F)<2\sum_{i=1}^n R_i\cdot |e_i|+\max\{0,\lambda(m+n)\}.\]

\end{prop}

To prove Proposition \ref{Irie and Seyfaddini}, we use the following proposition.



\begin{prop}\label{cL}
Let $W$ be an open subset of a $2w$-dimensional connected closed $\lambda$-monotone symplectic manifold $(\hat{W},\omega)$ and $\alpha\in[S^1,\hat{W}]$ a non-trivial homotopy class of free loops on $\hat{W}$.
Assume that a Hamiltonian function $H\colon \hat{W}\to\mathbb{R}$ satisfies the following conditions.
\begin{itemize}
\item for any point $x$ in $W$, $\phi_H^1(x)=x$ and $[\gamma_H^x]=-\alpha$,
\item $H$ is a Morse function and any contractible trajectory of the Hamiltonian isotopy $\{\phi_H^t\}$ is constant \textit{i.e.} $\mathrm{ev}(\mathcal{P}(H;0_{\hat{W}}))=\operatorname{Crit}(H)$,
\item $\operatorname{ind}_{\mathrm{Morse}}(x)=\operatorname{ind}_{\mathrm{CZ}}([x,c_x])$ for any point $x$ in $\mathrm{Crit}(H)$,
\end{itemize}
where $0_{\hat{W}}$ denotes the class of constant loops in $\hat{W}$.

Then for any Hamiltonian function $F\colon S^1\times W\to\mathbb{R}$ with compact support such that $\mathcal{P}(F;\alpha)=\emptyset$,
 \[c([\hat{W}],F)\leq 2||H||_{C^0}+\max\{0,\lambda w\},\]
where $||H||_{C^0}=\max_{x\in \hat{W}}|H(x)|$.

\end{prop}

The idea of using a Hamiltonian function $H$ satisfying the above conditions comes from Irie's paper \cite{Ir} (see also \cite{K}).
Seyfaddini's techniques of using the monotonicity assumption \cite{S} is also useful in our proof.
\begin{proof}
To give an upper bound of the spectral invariant associated to $F$, we consider the concatenation of $\phi_F^1$ and a Hamiltonian diffeomorphism $\phi_H^1$ with trajectories in $-\alpha$.
We can choose a smooth function $\chi\colon[0,\frac{1}{2}]\to[0,1]$ satisfying the following conditions.
\begin{itemize}
\item $\frac{\partial\chi}{\partial t}(t)\geq0$ for any $t\in{[0,\frac{1}{2}]}$, and
\item $\chi(t)=0$ for any $t\in[0,\frac{1}{5}]$ and $\chi(t)=1$ for any $t\in{[\frac{2}{5},\frac{1}{2}]}$.
\end{itemize}
Let  $L \colon S^1\times \hat{W}\to\mathbb{R}$ be a Hamiltonian function defined by
\begin{equation*}
L(t,x)=
\begin{cases}
\frac{\partial\chi}{\partial t}(t)H(\chi(t),x) & \text{when }t\in[0,\frac{1}{2}], \\
\frac{\partial\chi}{\partial t}(t-\frac{1}{2})F(\chi(t-\frac{1}{2}),x)& \text{when }t\in[\frac{1}{2},1].
\end{cases}
\end{equation*}
We claim
 \[c([\hat{W}],L)\leq ||H||_{C^0}+\max\{0,\lambda w\}.\]

Let $[z,u]$ be an element of $\tilde{\mathcal{P}}(H)$ and define $x$ by $x=\operatorname{ev}(z)$.
If $x\in W$, by the assumption on $H$, $[\gamma_H^x]=\mathcal{L}_{-\alpha}(W)$.
Since the path $\gamma_{L}^x$ is the concatenation of the paths $\gamma_H^x$ and $\gamma_F^{\phi_H(x)}$ up to parameter change, $\mathcal{P}(F;\alpha)=\emptyset$ implies 
$\gamma_L^x\notin \mathcal{L}_{0_{\hat{W}}}(\hat{W})$ for any $x\in W$.
If $x\notin W$, then $\phi_H(x)\notin  W$.
Thus $\gamma_L^x$ is equal to $\gamma_H^x$ up to parameter change and $\int_0^1H(t,\gamma_H^x(t))dt=\int_0^1L(t,\gamma_L^x(t))dt$.
Therefore we see that there exists a natural inclusion map $\iota\colon  \tilde{\mathcal{P}}(L)\to \tilde{\mathcal{P}}(H)$ which preserves values of the action functional and the Conley-Zehnder indices.

We give an estimate of the critical value of the action functional $\mathcal{A}_L$ attained by the fundamental class.
Since every element of $\mathcal{P}(H;0_{\hat{W}})$ is a constant loop, every element of $\mathcal{P}(L;0_{\hat{W}})$ is also a constant loop.
Since  $\mathcal{P}(L;0_{\hat{W}})$ is a finite set and $(\hat{W},\omega)$ is monotone, $\mathcal{A}_L(\tilde{\mathcal{P}}(H))$ is a discrete subset of $\mathbb{R}$.
For any element $[z,u]$ of $\tilde{\mathcal{P}}(L)$ which represents $\Phi([\hat{W}])$,  $\operatorname{ind}_{\mathrm{CZ}}([z,u])=w-2w=-w$.
Since every element of $\mathcal{P}(L;0_{\hat{W}})$ is a constant loop, there exists a point $x$ in $\hat{W}$ and $A\in\Gamma$ such that $\operatorname{ind}_{\mathrm{CZ}}([x,c_x\sharp A])=2w$ and $c([\hat{W}],L)=\mathcal{A}_L([x,c_x\sharp A])$.
Then, by the assumption,
\begin{align*}
&\operatorname{ind}_{\mathrm{Morse}}(x)+2c_1(A)\\
      &=(w-\operatorname{ind}_{\mathrm{CZ}}([x,c_x]))+2c_1(A)\\
      & =w-\operatorname{ind}_{\mathrm{CZ}}([x,c_x\sharp A])\\
      & =-w-(-w)=0.
\end{align*}
Since $0\leq\operatorname{ind}_{\mathrm{Morse}}(x)\leq2w$,
\[-w\leq c_1(A)\leq 0.\]
Since $\iota$ preserves values of the action functional,
\begin{align*}
\mathcal{A}_L([x,c_x\sharp A])&=\mathcal{A}_H([x,c_x\sharp A])\\
& =H(x)-\omega(A)\\
      & =H(x)-\lambda c_1(A),
\end{align*}
Thus, by $-w\leq c_1(A)\leq 0$, $c([\hat{W}],L)\leq||H||_{C^0}+\max\{0,\lambda w\}$.
By $||\bar{H}||_{C^0}=||H||_{C^0}$, the Lipschitz property and the homotopy invariance for spectral invariants (Proposition \ref{Oh} (1) and (2)) imply
\begin{align*}
c([\hat{W}],F)& \leq c([\hat{W}],L)+||\bar{H}||_{C^0}\\
&\leq(||H||_{C^0}+\max\{0,\lambda w\})+||H||_{C^0}\\
            & =2||H||_{C^0}+\max\{0,\lambda w\}.
\end{align*}
\end{proof}

To prove Proposition \ref{Irie and Seyfaddini}, we construct the Hamiltonian function $H$ in Proposition \ref{cL} by using $H^{R,\epsilon,e}$ given by the following lemma.
Let $0_T$ denote the free homotopy class of constant loops in $T^n_R\times T^n$.
\begin{lem}\label{function on torus}
Let $R$, $\epsilon$ be positive real numbers such that $2\epsilon <R$. 
Let $w_1$, $w_2$, $w_3$ and $w_4$ denote points $(R-\epsilon,0)$, $(R-\epsilon,\frac{1}{2})$, $(R+\epsilon,0)$ and $(R+\epsilon,\frac{1}{2})$ in $T^1_R\times T^1$, respectively.
For any non-zero integer $e$, there exists a Hamiltonian function $H^{R,\epsilon,e}\colon T^1_R\times T^1\to\mathbb{R}$ which satisfies the following conditions.

\begin{itemize}
\item  $H^{R,\epsilon,e}(p,q)=-ep$ on $U_\epsilon=(-R+2\epsilon,R-2\epsilon)\times T^1$,
\item $\operatorname{Crit} (H^{R,\epsilon,e})=\{w_1,w_2,w_3,w_4\}$,
\item $H^{R,\epsilon,e}$ is a Morse function,
\item $||H^{R,\epsilon,e}||_{L^\infty}<(R-\epsilon)\cdot |e|$,
\item $\operatorname{ev}(\mathcal{P}(H^{R,\epsilon,e};0_T))=\operatorname{Crit}(H^{R,\epsilon,e})$,
\item $\operatorname{ind}_{\mathrm{Morse}}(w_i)=\operatorname{ind}_{\mathrm{CZ}}([w_i,c_{w_i}])$ for any $i\in\{1,2,3,4\}$.
\end{itemize}
Here $\operatorname{Crit}(H^{R,\epsilon,e})$ is the set of critical points of $H^{R,\epsilon,e}$.
\end{lem}
\begin{proof}
We realize a 2-torus $T^2$ in $\mathbb{R}^3$ as
\[T^2=\{(x,y,z)\in\mathbb{R}^3; (\sqrt{x^2+z^2}-3)^2+y^2=1\}.\]
Define the (time-independent) Hamiltonian function $H\colon T^2\to\mathbb{R}$ by $H(x,y,z)=z$.
Note that the set of critical points of $H$ is 
\[\{(0,0,2), (0,0,4), (0,0,-2), (0,0,-4)\}.\]
We can take a diffeomorphism $f\colon T^1_R\times T^1\to T^2$ which maps $w_1, w_2, w_3$ and $w_4$ to $(0,0,2), (0,0,4)$, $(0,0,-2)$ and $(0,0,-4)$, respectively and satisfies $H(f(p,q))=\frac{p}{R}$ for any $p\in I_{R(2\epsilon)}$.
Let $u^{R,\epsilon,e}\colon \mathbb{R}\to\mathbb{R}$ be a function such that
\begin{itemize}
\item  $du^{R,\epsilon,e}(x)< 0$ for any real number $x$,
\item  $u^{R,\epsilon,e}(x)=-eRx$ if $|x|\leq 1-\frac{2\epsilon}{R}$,
\item $|u^{R,\epsilon,e}(x)|<(R-\epsilon)\cdot |e|$ if $|x|<4$,

\end{itemize}
Define the Hamiltonian function $H^{R,\epsilon,e}\colon T^1_R\times T^1\to\mathbb{R}$ by $H^{R,\epsilon,e}=u^{R,\epsilon,e}\circ H\circ f$.
Assume that $(du^{R,\epsilon,e})_x$ is sufficiently $C^1$-small for any $x$ with $2\leq |x|\leq 4$.
Then the Yorke estimate (\cite{Y}) implies $\operatorname{ev}(\mathcal{P}(H^{R,\epsilon,e};0_T))=\operatorname{Crit}(H^{R,\epsilon,e})$.
Since $2\leq|H(f(w_i))|\leq4$ for any $i$, $dH^{R,\epsilon,e}$ is sufficiently $C^1$-small near $\operatorname{Crit}(H^{R,\epsilon,e})$ and  hence $\operatorname{ind}_{\mathrm{Morse}}(w_i)=\operatorname{ind}_{\mathrm{CZ}}([w_i,c_{w_i}])$ for any $i$.
\end{proof}

\begin{proof}[Proof of Proposition \ref{Irie and Seyfaddini}]
When $e=0$, our proposition immediately follows from the Arnold conjecture.
Thus we may assume $e\neq0$.

To use Proposition \ref{cL}, we construct the Hamiltonian function $H$.
Define the Hamiltonian function $H^\prime\colon T^n_R\times T^n\to\mathbb{R}$ by
\[H^\prime(p,q)=\sum_{i=1}^nH^{R_i,\epsilon_i,e_i}(p_i,q_i).\]
Then $\gamma_{H^\prime}^x \in \mathcal{L}_{-e} (T^n_R\times T^n)$ for any $x\in U_\epsilon$.
Thus we can take a neighborhood $W$ of $U_\epsilon$ such that
\[\operatorname{ev}(\mathcal{P}(H^\prime;(0_M,0_T)))\cap \bar{W}=\emptyset.\]
In order to compute the spectral invariant, we take a non-degenerate perturbation of $H^\prime$.
Let $\rho\colon T^n_R\times T^n\to[0,1]$ be a function such that
\begin{equation*}
\rho(p,q)=
\begin{cases}
1 & \text{for any }(p,q)\in T^n_R\times T^n\setminus W, \\
0 & \text{for any }(p,q)\in U_\epsilon.
\end{cases}
\end{equation*}
Let $G\colon M\to\mathbb{R}$ be a Morse function and define the Hamiltonian function $H\colon M\times T^n_R\times T^n\to\mathbb{R}$ by
\[H(y,p,q)=H^\prime(p,q)+\rho(p,q)\cdot G(y).\]
If the Morse function $G$ is sufficiently $C^2$-small, then 
\begin{itemize}
\item $\operatorname{ev}(\mathcal{P}(H;(0_M,0_T)))\cap(M\times W)=\emptyset,$ and
\item there exist only finitely many points $y_1, \ldots, y_k$ in $M$ such that $\operatorname{Crit}(G)=\operatorname{ev}(\mathcal{P}(tG;0_M))=\{y_1, \ldots, y_k\}$ for any $t\in(0,1]$.
\end{itemize}
  Thus
  \[\operatorname{ev}(\mathcal{P}(H;(0_M,0_T)))=\{(y_i,(w_{j_1},\ldots,w_{j_n}))\}_{i\in\{1,\ldots,k\}, j_1,\ldots,j_n\in\{1,2,3,4\}}=\operatorname{Crit}(H).\]
  By Lemma \ref{function on torus}, 
  \[\operatorname{ind}_{\mathrm{Morse}}(x)=\operatorname{ind}_{\mathrm{CZ}}([x,c_x]),\]
  for any point $x$ in $\operatorname{Crit}(H)$.

Hence $H$ satisfies the conditions of Proposition \ref{cL} and thus we apply Proposition \ref{cL}.

By Proposition \ref{cL} and $||\bar{H}||_{C^0}=||H||_{C^0}$, the Lipschitz property and the homotopy invariance for spectral invariants (Proposition \ref{Oh} (1) and (2)) imply
\begin{align*}
c([M\times T^n_R\times T^n],F)& \leq 
2||H||_{C^0}+\max\{0,\lambda(m+n)\}\\
            & <2(\sum_{i=1}^n(R_i-\epsilon)\cdot |e_i|+||G||_{C^0})+\max\{0,\lambda(m+n)\},
\end{align*}
If the Morse function $G$ is sufficiently $C^2$-small,
\[c([M\times T^n_R\times T^n],F)<2\sum_{i=1}^nR_i\cdot |e_i|+\max\{0,\lambda(m+n)\}.\]
\end{proof}

To prove Theorem \ref{new characterization of heaviness}, we use the following theorems by Entov and Polterovich (\cite{EP09}).
\begin{thm}[{\cite{EP09} Theorem 1.7}]\label{product of heavy}
Let $(N_1,{\omega}_1)$ and $(N_2,{\omega}_2)$ be closed symplectic manifolds.
Assume that for $i=1,2$, $Y_i$ is a heavy subset of $(N_i,\omega_i)$.
Then the product $Y_1{\times}Y_2$ is a heavy subset of $N_1{\times}N_2$.
\end{thm}

\begin{thm}[{\cite{EP09} Theorem 1.4}]\label{fundamental heavy}
Let $(N,\omega)$ be a closed symplectic manifold.
Assume that $Y$ is a heavy subset of $(N,\omega)$.
Then $Y$ is $[N]$-heavy.
\end{thm}

\begin{proof}[Proof of Theorem \ref{new characterization of heaviness}]
Fix a positive real number $\epsilon$ such that $\epsilon <\min\{R_1,\ldots,R_n\}$ and take a Hamiltonian function $F\colon S^1\times M\times I^n_{R(\epsilon)}\times T^n\to\mathbb{R}$ with compact support such that $F|_{S^1\times X\times T^n}\geq 2\sum_{i=1}^nR_i\cdot |e_i|+\max\{0,\lambda(m+n)\}$.
Assume $\mathcal{P}(F;(0_M,e))= \emptyset$.
Then, Proposition \ref{Irie and Seyfaddini} and the triangle inequality imply
\[\zeta_{[M\times T^n_{R(\epsilon)}\times T^n]}(F)< 2\sum_{i=1}^nR_i\cdot |e_i|+\max\{0,\lambda(m+n)\}.\]

Note that Example \ref{torus} and Theorem \ref{product of heavy} imply that $X\times T^n$ is a heavy subset.
Since Theorem \ref{fundamental heavy} implies that $X\times T^n$ is $[M\times T^n_{R(\epsilon)}\times T^n]$-heavy,  by Definition \ref{definition of heavy},
\[\zeta_{[M\times T^n_{R(\epsilon)}\times T^n]}(F)\geq 2\sum_{i=1}^nR_i\cdot |e_i|+\max\{0,\lambda(m+n)\}.\]

These two inequalities contradicts.
Since any Hamiltonian function $F\colon S^1\times M\times I^n_R\times T^n\to\mathbb{R}$ with compact support has support in $S^1\times M \times I^n_{R(\epsilon)}\times T^n$ for some $\epsilon$, we complete the proof of Theorem \ref{new characterization of heaviness}.
\end{proof}

As we mentioned in Introduction, Theorem \ref{new characterization of heaviness} gives the inequality 
\[C(\mathbb{C}P^m,\Phi^{-1}(y_0),R;e)\leq2\sum_{i=1}^mR_i\cdot |e_i|.\]

We have another example.
\begin{exam}\label{torus2}
Since $\pi_2(T_R^n\times T^n)=0$, by applying Theorem \ref{capacity statement} to Example \ref{torus}, we attain the inequality $C(T_R^n\times T^n,T^n,R;e)\leq2\sum_{i=1}^mR_i\cdot |e_i|$ for any elements $e=(e_1,\ldots,e_m)$ and $R=(R_1,\ldots,R_m)$ of $\mathbb{Z}^m$ and $(\mathbb{R}_{>0})^m$, respectively.
\end{exam}

\section{Non-contractible trajectories on non-monotone symplectic manifolds}\label{non-monotone}

Some reserches studied the existence problem of non-contractible periodic orbits whose period is not more than 1 (for example, see \cite{GL} and \cite{L}).
When we replace the existence problem of 1-periodic trajectories by  the existence problem of periodic orbits whose period is not more than 1, we have the following result which does not need the assumption of monotonicity.

\begin{thm}\label{new characterization of heaviness2}
Let $X$ be a heavy subset of a connected closed symplectic manifold $(M,\omega)$.
Let $e=(e_1,\ldots,e_n)$ and $R=(R_1,\ldots,R_n)$ be elements of $\mathbb{Z}^n$ and $(\mathbb{R}_{>0})^n$, respectively.
For any (time-independent) Hamiltonian function $F\colon M\times I_R^n\times T^n\to \mathbb{R}$ with compact support such that $F|_{X\times T^n}\geq 2\sum |e_i| R_i$,
the Hamiltonian flow $\{\phi_F^{t}\}_{t\in\mathbb{R}}$ has a periodic orbits $(1,e)$ whose period is not more than 1 in the free loop homotopy class $(0_M,e)$. 
\end{thm}

To prove Theorem \ref{new characterization of heaviness2}, we give an upper bound of the spectral invariant for a Hamiltonian function $F\colon S^1\times M\times U_\epsilon\to\mathbb{R}$ such that its Hamiltonian isotopy $\{\phi_F^t\}$ has no trajectories in the free loop homotopy class $(0_M,e)$ whose period is not more than 1.
For $R=(R_1,\ldots,R_n)\in (\mathbb{R}_{>0})^n$ and a positive real number $\epsilon$ with  $\epsilon <\min\{R_1,\ldots,R_n\}$, let $R(\epsilon)$ denote $(R_1-\epsilon,\ldots,R_n-\epsilon)\in (\mathbb{R}_{>0})^n$, as before.

\begin{prop}\label{Irie and Seyfaddini2}
Let $(M,\omega)$ be a connected closed symplectic manifold.
Let $e=(e_1,\ldots,e_n)$ and $R=(R_1,\ldots,R_n)$ be elements of $\mathbb{Z}^n$ and $(\mathbb{R}_{>0})^n$, respectively.
For a positive real number $\epsilon$ with $2\epsilon <\min\{R_1,\ldots,R_n\}$, we define the open subset $U_\epsilon$ of $T^n_R\times T^n$ as in Proposition \ref{Irie and Seyfaddini}.
Then for any Hamiltonian function $F\colon S^1\times M\times U_\epsilon\to\mathbb{R}$ with compact support such that its Hamiltonian isotopy $\{\phi_F^t\}$ has no trajectories in the free loop homotopy class $(0_M,e)$ whose period is not more than 1,
\[c([M\times T^n_R\times T^n],F)<2\sum_{i=1}^nR_i\cdot |e_i|.\]
\end{prop}

To prove Proposition \ref{Irie and Seyfaddini2}, we use the following proposition which is a modification of an argument in \cite{Ir}.
\begin{prop}\label{cL2}
Let $W$ be an open subset of a $2w$-dimensional connected closed symplectic manifold $(\hat{W},\omega)$ and $\alpha\in[S^1,\hat{W}]$ a non-trivial homotopy class of free loops on $\hat{W}$.
Assume that a Hamiltonian function $H\colon W\to\mathbb{R}$ satisfies the following conditions.
\begin{itemize}
\item for any point $x$ in $W$, $\phi_H^1(x)=x$ and $[\gamma_H^x]=-\alpha$,
\item $H$ is non-degenerate,
\end{itemize}
Let $0_{\hat{W}}$ denote the class of constant loops in $\hat{W}$.

Let $F\colon S^1\times W\to\mathbb{R}$ be a Hamiltonian function with compact support such that its Hamiltonian isotopy $\{\phi_F^t\}$ has no trajectories in the free loop homotopy class $(0_{\hat{W}},e)$ whose period is not more than 1.
Then
 \[c([\hat{W}],F)\leq 2||H||_{C^0}.\]
\end{prop}

As with Proposition \ref{cL}, the idea of using a Hamiltonian function $H$ satisfying the above conditions comes from Irie's paper \cite{Ir} (see also \cite{K}).

For a Hamiltonian function $H\colon S^1\times M\to\mathbb{R}$ with compact support,
let $\operatorname{Spec}(H)$ denote the set of critical values of the action functional $\mathcal{A}_H$ $\textit{i.e.}$ $\mathcal{A}_H(\tilde{\mathcal{P}}(H))$.
To prove Proposition \ref{cL2}, we use the following theorem.
\begin{thm}[\cite{U08}, \cite{Oh09}, non-degenerate spectrality]\label{non-deg spec}
Let $(M,\omega)$ be a closed symplectic manifold and $a$ be an element of $QH_\ast(M,\omega)$.
Then for any non-degenerate Hamiltonian function $F\colon S^1\times M\to\mathbb{R}$ with compact support,
$c(a,F)\in \operatorname{Spec}(F)$.
\end{thm}

\begin{proof}[Proof of Proposition \ref{cL2}]
We give an upper bound of the spectral invariant associated to $F$ by using the concatenation with $\phi_H^t$.

For a real number $s$ with $s \in [0,1]$, we define the new Hamiltonian function $L^s \colon S^1\times\hat{W}\to\mathbb{R}$ as follows:
\begin{equation*}
L^s(t,x)=
\begin{cases}
\frac{\partial\chi}{\partial t}(t)H(\chi(t),x) & \text{when }t\in[0,\frac{1}{2}], \\
s\frac{\partial\chi}{\partial t}(t-\frac{1}{2})F(s\chi(t-\frac{1}{2}),x)& \text{when }t\in[\frac{1}{2},1],
\end{cases}
\end{equation*}
where $\chi\colon[0,\frac{1}{2}]\to[0,1]$ is the function defined in the proof of Proposition \ref{Irie and Seyfaddini}.
Since $\frac{\partial\chi}{\partial t}=0$ on neighborhoods of $t=0$ and $t=\frac{1}{2}$, $L^s$ is a smooth Hamiltonian function.

We claim  $\operatorname{Spec}(L^s)\subset\operatorname{Spec}(H)$ for a real number $s$ with $s \in [0,1]$.
Let $F^s\colon S^1\times\hat{W}\to\mathbb{R}$ denote the Hamiltonian function defined by $F^s(t,x)=s\frac{\partial\chi}{\partial t}(\frac{t}{2})F(s\chi(\frac{t}{2}),x)$.
Let $[z,u]$ be an element of $\tilde{\mathcal{P}}(H)$ and define $x$ by $x=\operatorname{ev}(z)$.
If $x\in W$, by the definition of $H$, $\gamma_H^x\in\mathcal{L}_{0_{\hat{W}}}(W)$.
Since the path $\gamma_{L^s}^x$ is the concatenation of the paths $\gamma_H^x$ and $\gamma_{F^s}^{\phi_H(x)}$ up to parameter change and $\{\phi_F^t\}$ has no trajectories in the free loop homotopy class $0_{\hat{W}}$ whose period is not more than 1,
$\gamma_{L^s}^x\notin \mathcal{L}_{0_{\hat{W}}}(\hat{W})$ for any $x\in W$.
If $x\notin W$, then $\phi_H(x)\notin W$.
Thus $\gamma_{L^s}^x$ is equal to $\gamma_H^x$ up to parameter change and $\int_0^1H(t,\gamma_H^x(t))dt=\int_0^1L(t,\gamma_{L^s}^x(t))dt$.
Therefore we see that there exists a natural inclusion map $\iota\colon  \tilde{\mathcal{P}}(L^s)\to \tilde{\mathcal{P}}(H)$ which preserves values of the action functional, and hence $\operatorname{Spec}(L^s)\subset\operatorname{Spec}(H)$.
Since $H$ is a non-degenerate Hamiltonian function, $L^s$ is also non-degenerate, and hence Theorem \ref{non-deg spec} implies $c([\hat{W}],L^s)\in\operatorname{Spec}(H)$.

By the Lipschitz property for spectral invariants (Proposition \ref{Oh} (1)), $c([\hat{W}],L^s)$ depends continuously on $s$. 
Since $\operatorname{Spec}(H)$ is a measure-zero set (Lemma 2.2 of \cite{Oh02}),  $c([\hat{W}],L^s)$ is a constant function of $s$.
The homotopy invariance for spectral invariants (Proposition \ref{Oh} (2)) implies 
\[c([\hat{W}],L^0)=c([\hat{W}],H).\]
Hence for any $s \in [0,1]$,
 \[c([\hat{W}],L^s)=c([\hat{W}],H).\]
Then $c([\hat{W}],F)$ is estimated as follows.
 \begin{align*}
c([\hat{W}],F)& \leq c([\hat{W}],L^1)+||\bar{H}||_{C^0}\\
& = c([\hat{W}],H)+||H||_{C^0}\\
      & <2(\sum_{i=1}^n(R_i-\epsilon)\cdot |e_i|+||G||_{C^0}).
\end{align*}
\end{proof}

\begin{proof}[Proof of Proposition \ref{Irie and Seyfaddini2}]
Let $G$ be a Morse function on $M$ and $H\colon \hat{W}\to\mathbb{R}$ the Hamiltonian function defined in the proof of Proposition \ref{Irie and Seyfaddini}.

As we explained in the proof of Proposition \ref{Irie and Seyfaddini}, if the Morse function $G$ is sufficiently $C^2$-small, then 
  \[\operatorname{ev}(\mathcal{P}(H;(0_M,0_T)))=\{(y_i,(w_{j_1},\ldots,w_{j_n}))\}_{i\in\{1,\ldots,k\}, j_1,\ldots,j_n\in\{1,2,3,4\}}=\operatorname{Crit}(H).\]
In particular, $H$ is a non-degenerate Hamiltonian function.
Since 
\[||\bar{H}||_{C^0}=||H||_{C^0}\leq \sum_{i=1}^n(R_i-\epsilon)\cdot |e_i|+||G||_{C^0},\]
Proposition \ref{cL2} implies
\[c([M\times T^n_R\times T^n],F)<2(\sum_{i=1}^n(R_i-\epsilon)\cdot |e_i|+||G||_{C^0}).\]
If the Morse function $G$ is sufficiently $C^2$-small,
\[c([M\times T^n_R\times T^n],F)<2\sum_{i=1}^nR_i\cdot |e_i|.\]
\end{proof}

\begin{proof}[Proof of Theorem \ref{new characterization of heaviness2}]
Fix a positive real number $\epsilon$  with $\epsilon <\min\{R_1,\ldots,R_n\}$ and take a Hamiltonian function $F\colon S^1\times M\times I^n_{R(\epsilon)}\times T^n\to\mathbb{R}$ with compact support such that $F|_{S^1\times X\times T^n}\geq 2\sum_{i=1}^nR_i\cdot |e_i|$.
Assume that $\{\phi_F^t\}$ has no trajectories in the free loop homotopy class $(0_M,e)$ whose period is not more than 1.
Since  $\{\phi_F^t\}$ has no trajectories in the free loop homotopy class $(0_M,e)$ whose period is not more than 1, Proposition \ref{Irie and Seyfaddini2} and the triangle inequality for spectral invariants (Proposition \ref{Oh} (3)) imply $\zeta_{[M\times T^n_{R(\epsilon)}\times T^n]}(F)< 2\sum_{i=1}^nR_i\cdot |e_i|$.

 By applying Theorem \ref{product of heavy} to Example \ref{torus}, we see that $X\times T^n$ is a heavy subset.
Then Proposition \ref{fundamental heavy} implies that $X\times T^n$ is $[M\times T^n_{R(\epsilon)}\times T^n]$-heavy, and hence $\zeta_{[M\times T^n_{R(\epsilon)}\times T^n]}(F)\geq\inf_{X\times T^n}F\geq 2\sum_{i=1}^nR_i\cdot |e_i|$ by Definition \ref{definition of heavy}.

These two inequalities contradict and we proved that  $\{\phi_F^t\}$ has a trajectory in the free loop homotopy class $(0_M,e)$ whose period is not more than 1.
Since any Hamiltonian function $F\colon S^1\times M\times I^n_R\times T^n\to\mathbb{R}$ with compact support has support in $S^1\times M \times I^n_{R(\epsilon)}\times T^n$ for some $\epsilon$, we complete the proof of Theorem \ref{new characterization of heaviness2}.
\end{proof}

\section{Displaceable subsets and non-contractible trajectories}\label{displaceable}

For a Hamiltonian function $H\colon S^1\times M\to \mathbb{R}$ with compact support on a symplectic manifold $M$, we define \textit{the Hofer length} $||H||$ of $H$ by
\[||H||=\int_0^1||H_t||_{C^0}dt.\]
For a subset $X$ of $M$, we define \textit{the displacement energy} of $X$ by
\[E(X)=\inf\{ ||H|| ; H\in C_c^\infty(S^1\times M), \bar{X}\cap\phi_H^1(X)=\emptyset \} ,\]
where $\bar{X}$ is the topological closure of $X$.
If $X$ is non-displaceable, we define $E(X)=\infty$.

\begin{proof}[Proof of Proposition \ref{our displaceability}]
To use Proposition \ref{BPS displaceability}, we construct the Hamiltonian function $\hat{H}\colon S^1\times M\times I_R^n\times T^n\to\mathbb{R}$ such that $(X\times T^n)\cap\phi^1_{\hat{H}}(X\times T^n)=\emptyset$ and $\mathcal{P}(\hat{H};(0_M,e))=\emptyset$.
Fix a sufficiently small positive number $\epsilon$.
We take a Hamiltonian function $H\colon S^1\times M\to\mathbb{R}$ with compact support such that $||H||<E(X)+\epsilon$ and $X\cap \phi_H^1(X)=\emptyset$.
Since $|e_k|\cdot R_k>E(X)$ and $\epsilon$ is sufficiently small, we can take a function $\rho_k\colon (-R_k,R_k)\to\mathbb{R}$ with compact support and such that
\begin{itemize}
\item $\rho_k\equiv 1$ in a neighborhood of \{0\},
\item $|\dot{\rho}_k(x)|<|e_k|\cdot (E(X)+\epsilon)^{-1}$ for any $x\in (-R_k,R_k)$.
\end{itemize}
For $i\neq k$, we take a function $\rho_i \in C_c^\infty(-R_i,R_i)$ with $\rho_i\equiv 1$ in a neighborhood of $\{0\}$.
we define the Hamiltonian function $\hat{H}\colon S^1\times M\times I_R^n\times T^n\to\mathbb{R}$ by
\[\hat{H}(t,x,p,q)=\prod_i\rho_i(p_i)\cdot H(t,x).\]
Then
\[(X_{\hat{H}}^t)_{(x,p,q)}=(\prod_i\rho_i(p_i)\cdot (X_H^t)_x,0,\ldots,0,\dot{\rho}_1(p_1)\cdot H(t,x),\ldots,\dot{\rho}_n(p_n)\cdot H(t,x) ).\]
Since  $\rho_i\equiv 1$ in a neighborhood of $\{0\}$, $(X\times T^n)\cap\phi^1_{\hat{H}}(X\times T^n)=\emptyset$.
Since $|\dot{\rho}_k|<|e_k|\cdot (E(X)+\epsilon)^{-1}$ and  $\int_0^1||H_t||_{C^0}dt=||H||<E(X)+\epsilon$, $\int_0^1|\dot{\rho}_k(p_k)|\cdot |H(t,x)|dt=\dot{\rho}_k(p_k)\cdot|\int_0^1|H(t,x)|dt$ is smaller than $|e_k|$ and hence $\mathcal{P}(\hat{H};(0_M,e))=\emptyset$.
Thus Proposition \ref{BPS displaceability} implies 
\[C(M,X,R;e)=C_{BPS}(M\times I_R^n\times T^n, X\times T^n;(0_M,e))=\infty.\]
\end{proof}

\section{Compressible Hamiltonian torus action and non-contractible trajectories}\label{compressible section}
We have a family of examples similar to Example \ref{torus and cp^n}.
Let $(M,\omega)$ be a closed symplectic manifold.
We consider the case when a moment map $\Phi=(F^1,\ldots,F^l)\colon M\to\mathbb{R}^l$ induces a Hamiltonian torus action $\textit{i.e.}$ $\phi_{F^i}^1=\mathrm{id}$ for $i=1,\ldots,k$ and $\{F^i,F^j\}=0$ for $i\neq j$.
Then there exists a natural inclusion map $\iota\colon T^l\to\mathrm{Ham}(M,\omega)$.
A Hamiltonian action induced by $\Phi$ is \textit{compressible} if the image of the map $\iota_\ast\colon \pi_1(T^l)\to\pi_1(\mathrm{Ham}(M,\omega))$ is a finite group, where $\iota_\ast$ is a homomorphism induced by $\iota$.

Entov and Polterovich proved the following theorem.
\begin{thm}[\cite{EP09}]\label{EP compressible}
Let $(M,\omega)$ be a $2m$-dimensional connected closed symplectic manifold and $\Phi=(F^1,\ldots,F^l)\colon M\to\mathbb{R}^l$ a moment map which induces a compressible Hamiltonian torus action.
Assume that  $F^i$ is normalized as a Hamiltonian function for any $i$.
Then
\begin{description}
\item[(1)]  $\Phi^{-1}(0)$ is heavy, thus stably non-displaceable,
\item[(2)]  $\Phi^{-1}(y)$ is stably displaceable for any point $y$ in $\Phi(M)$ with $y\neq 0$.
\end{description}
\end{thm}

We have the corresponding result on the existence problem of non-contractible trajectories.
\begin{thm}\label{compressible}
Let $(M,\omega)$ be a connected closed $\lambda$-monotone symplectic manifold and $\Phi=(F^1,\ldots,F^l)\colon M\to\mathbb{R}^l$ be a moment map which induces a compressible Hamiltonian torus action.
Assume that  $F^i$ is normalized as a Hamiltonian function for any $i$.
Then there exists a positive real number $E$ such that
\begin{description}
\item[(1)]  $C(M,\Phi^{-1}(0),R;e)\leq 2\sum_{i=1}^nR_i\cdot |e_i|+\max\{0,-\lambda(m+n)\}$ for any elements $e=(e_1,\ldots,e_n)$ and $R=(R_1,\ldots,R_n)$ of $\mathbb{Z}^n$ and $(\mathbb{R}_{>0})^n$, respectively,
\item[(2)]  $C(M,\Phi^{-1}(y),R;e)=\infty$ for any point $y$ in $\Phi(M)$ with $y\neq 0$ and for any elements $e=(e_1,\ldots,e_n)$ and $R=(R_1,\ldots,R_n)$ of $\mathbb{Z}^n$ and $(\mathbb{R}_{>0})^n$ with $R_k>E$ and $e_k\neq 0$ for some $k$, respectively.
\end{description}
\end{thm}

(1) of Theorem \ref{compressible} follows immediately from Theorem \ref{capacity statement} and (1) of Theorem \ref{EP compressible}.

To prove (2) of Theorem \ref{compressible}, we use the following theorem which is a slight modification of  Theorem 2.1 of \cite{EP09}.
Note that we can identify $T^\ast T^1$ with $\mathbb{R}\times T^1$ with coordinates $(p,q)$.
\begin{prop}\label{how to stably disp}
Let $X$ be a compact subset of a closed symplectic manifold $M$.
Assume that there exists a normalized Hamiltonian function $F\colon S^1\times M\to\mathbb{R}$ generating a loop $\{\phi_F^t\}_{t\in[0,1]}$ in $\mathrm{Ham}(M,\omega)$ which is homotopic to the trivial isotopy relative to endpoints and $F(t,x)\neq 0$ for any $t$ and any point $x$ with $x\notin X$.
Then there exists a Hamiltonian function $H\colon S^1\times M\times T^\ast T^1\to\mathbb{R}$ with compact support such that $(X\times T^1)\cap\phi_{H}^1(X\times T^1)=\emptyset$ and $|\frac{\partial H}{\partial p}(t,x,p,q)|<1$ for any point $(t,x,p,q)$ in $S^1\times M\times T^\ast T^1$.
\end{prop}
Entov and Polterovich originally constructed a Hamiltonian function $\hat{H}$ such that $(X\times T^1)\cap\phi_{\hat{H}}^1(X\times T^1)=\emptyset$ with non-compact support.
We construct a Hamiltonian function $H$ serving our purpose by multiplying a bump function.

\begin{proof}
Let $\{f_t^s\}_{s,t\in[0,1]}$ be a homotopy of loop $\{\phi_F^s\}_{s\in[0,1]}$ to the constant loop \textit{i.e.} $f_0^s=\mathrm{id}$ and $f_1^s=\phi_F^s$.
Let $F^t\colon S^1\times M\to\mathbb{R}$ denote the normalized Hamiltonian function generating the isotopy $\{f_t^s\}_{s\in[0,1]}$.
Consider the family of diffeomorphisms $\Psi_t$ of $M\times T^\ast T^1$ defined by 
\[\Psi_t(x,p,q)=(f_t^qx,p-F^t(q,f_t^qx),q).\]
By Theorem 6.1.B of \cite{P01}, $\Psi_t$ is a Hamiltonian isotopy.
Let $\hat{H}\colon S^1\times M\times T^\ast T^1\to\mathbb{R}$ be a Hamiltonian function generating $\Psi_t$.
Note that $\hat{H}$ does not depend on the coordinate $p$ since $\operatorname{pr}_3(\frac{d\Psi_t}{ds})=0$,
where $\operatorname{pr}_3\colon M\times T^\ast T^1\to T^1$ is the projection defined by $\operatorname{pr}_3(x,p,q)=q$.

We can take a function $\rho\colon \mathbb{R}\to\mathbb{R}$ with compact support such that
\begin{itemize}
\item $\rho= 1$ in $\bigcup_t(\operatorname{pr}_2(\Psi_t(X)))$, where $\operatorname{pr}_2\colon M\times T^\ast T^1\to\mathbb{R}$ is the projection defined by $\operatorname{pr}_2(x,p,q)=p$,

\item $|\dot{\rho}(x)|<\inf_t||\hat{H_t}||_{C^0}^{-1}$ for any $x\in \mathbb{R}$.
\end{itemize}
Let $H\colon S^1\times M\times T^\ast T^1\to\mathbb{R}$ a Hamiltonian function defined by
\[H(t,x,p,q)=\rho(p)\cdot\hat{H}(t,x,p,q).\]
Since $\hat{H}$ does not depend on the coordinate $p$,
\[{\operatorname{pr}_3}_\ast((X_H^t)_{(x,p,q)})=\dot\rho(p)\cdot\hat{H}_t(x,p,q).\]
Since $|\dot{\rho}(x)|<\inf_t||\hat{H_t}||^{-1}$, $|\frac{\partial H}{\partial p}(t,x,p,q)|<1$.
Since $\rho= 1$ in $\bigcup_t(\operatorname{pr}_2(\Psi_t(X)))$, $(X\times T^1)\cap\phi_{H}^1(X\times T^1)=\emptyset$.

\end{proof}
The construction of $\Psi_t$ appeared in \cite{P01} and \cite{EP09}.

\begin{proof}[Proof of (2) of Theorem \ref{compressible}]
Let $e$ be an element of $(\mathbb{Z}_{>0})^n$ with $e_k\neq 0$.
To use Proposition \ref{BPS displaceability}, we construct a Hamiltonian function $\hat{H}\colon S^1\times M\times I_R^n\times T^n\to\mathbb{R}$ such that 
\[(\Phi^{-1}(y)\times T^n)\cap\phi^1_{\hat{H}}(\Phi^{-1}(y)\times T^n)=\emptyset,\]
and $\mathcal{P}(\hat{H};(0_M,e))=\emptyset$.

First, we prepare some Hamiltonian functions $H^k\colon S^1\times M\times T^\ast T^1\to\mathbb{R}$ ($k=1,\ldots,l$).
Since the action induced by $\Phi$ is compressible,  for any $k$ there exists a sufficient large positive integer $N_k$ such that the Hamiltonian function $N_kF^k$ generates a contractible Hamiltonian circle action on $M$.

Since $N_kF^k$ generates a contractible Hamiltonian circle action on $M$ for any $k$, Proposition \ref{how to stably disp} implies that there exist Hamiltonian functions $H^k\colon S^1\times M\times T^\ast T^1\to\mathbb{R}$ ($k=1,\ldots,l$) with compact support such that 
\[(\Phi^{-1}(y)\times T^1)\cap\phi_{H^k}^1(\Phi^{-1}(y)\times T^1)=\emptyset\]
for any $y$ with $y_k\neq 0$ and $|\frac{\partial H^k}{\partial p}(x,p,q)|<1$ for any point $(x,p,q)$ in $M\times T^\ast T^1$.

Define the projection $\operatorname{pr}_2\colon M\times T^\ast T^1\to\mathbb{R}$ by $\operatorname{pr}_2(x,p,q)=p$ and
put $E=\max_k\sup\{|r|;r\in\operatorname{pr}_2(\cup_{t\in[0,1]}\operatorname{Supp}(H_t^k))\}$.

Fix a point $y=(y_1,\ldots,y_l)$ of $\Phi(M)$ with $y\neq 0$.
Then, there exists some $k$ such that $y_k\neq0$.
Let $R=(R_1,\ldots,R_n)$ be an element of $(\mathbb{R}_{>0})^n$ with $R_k>E$.
For $i\neq k$, we take a function $\rho_i \colon(-R_i,R_i)\to[0,1]$ with compact support such that $\rho_i= 1$ in a neighborhood of $\{0\}$.
Let $\hat{H}^k\colon S^1\times M\times I_R^n\times T^n\to\mathbb{R}$ be a Hamiltonian function defined by
\[\hat{H}^k(t,x,p,q)=\prod_{i\neq k}\rho_i(p_i)\cdot H^k(t,x,p_k,q_k).\]
Since $R_k>E$, $\hat{H}^k$ has compact support in $S^1\times M\times I_R^n\times T^n$.
Then
\[\operatorname{pr}_\ast((X_{\hat{H}^k}^t)_{(x,p,q)})=\prod_{i\neq k}\rho_i(p_i)\cdot\frac{\partial H^k}{\partial p}(t,x,p_k,q_k),\]
where $\operatorname{pr}\colon M\times I_R^n\times T^n\to T^1$ is the projection defined by $\operatorname{pr}(x,p,q)=q_k$.

Since 
\[(\Phi^{-1}(y)\times T^1)\cap\phi_{H^k}^1(\Phi^{-1}(y)\times T^1)=\emptyset,\]
and $\rho_i= 1$ in a neighborhood of $\{0\}$, 
\[(\Phi^{-1}(y)\times T^n)\cap\phi^1_{\hat{H}^k}(\Phi^{-1}(y)\times T^n)=\emptyset.\] 
Since  $|\frac{\partial H^k}{\partial p}(t,x,p,q)|<1$ for any point $(t,x,p,q)$ in $S^1\times M\times T^\ast T^1$ and the image of $\rho_i$ is in $[0,1]$, $\mathcal{P}(\hat{H}^k;(0_M,e))=\emptyset$.
Thus Proposition \ref{BPS displaceability} implies 
\[C(M,X,R;e)=C_{BPS}(M\times I_R^n\times T^n, X\times T^n;(0_M,e))=\infty,\]
for any $e$ with $e\neq 0$.
\end{proof}

\section{Counter example for displaceable subsets}\label{conj section}

We cannot replace the assumption that $X$ is $n$-stably non-displaceable in  Conjecture \ref{sn conj} by that $X$ is non-displaceable.
We have the following example.
\begin{prop}\label{counter for non-disp}
Let $S^2$ be the 2-sphere $\{(x,y,z)\in\mathbb{R}^3;x^2+y^2+z^2=1\}$ with the standard area (symplectic) form.
For a positive real number $\epsilon$, we define a subset $C_\epsilon$ of $S^2$ by $C_\epsilon=\{(x,y,z)\in S^2;z=\pm \epsilon\}$.
Then $C_\epsilon$ is stably displaceable for any positive real number $\epsilon$ and there exists a positive real number $E$ such that
\[C(S^2,C_\epsilon,R;e)=\infty,\]
for any positive real number $\epsilon$ and any elements $R$ and $e$ of $(\mathbb{R}_{>0})^n$ and  $\mathbb{Z}^n$  with $R_k>E$ and $e_k\neq 0$ for some $k$, respectively.
\end{prop}

\begin{rem}
Let $A_\epsilon$ and $B_\epsilon$ be the subsets of $S^2$ defined by $A_\epsilon=\{(x,y,z)\in S^2;|z|\leq\epsilon\}$ and $B_\epsilon=\{(x,y,z)\in S^2;z>\epsilon\}$, respectively.
If $\epsilon<\frac{1}{3}$, then $\operatorname{Area}(A_\epsilon)<\operatorname{Area}(B_\epsilon)$.
Since any Hamiltonian diffeomorphism is area-preserving, $C_\epsilon$ is non-displaceable.

Professor Kaoru Ono told the author that $C_\epsilon$ for $\epsilon<\frac{1}{3}$ is an example due to Professor Leonid Polterovich of a non-displaceable subset which is stably displaceable.
\end{rem}

\begin{proof}[Proof of Proposition \ref{counter for non-disp}]
Let $e$ be an element of $(\mathbb{Z}_{>0})^n$ with $e_k\neq 0$.
To use Proposition \ref{BPS displaceability}, we construct a Hamiltonian function $\hat{H}\colon S^1\times S^2\times I_R^n\times T^n\to\mathbb{R}$ such that $(C_\epsilon\times T^n)\cap\phi^1_{\hat{H}}(C_\epsilon\times T^n)=\emptyset$ and $\mathcal{P}(\hat{H};(0_M,e))=\emptyset$.

Let $F\colon S^2\to\mathbb{R}$ be a Hamiltonian function defined by $F(x,y,z)=4\pi z$.
The isotopy $\{\phi_F^t\}_{t\in[0,1]}$ is homotopic to the trivial isotopy relative to endpoints.

Thus Proposition \ref{how to stably disp} implies that  there exists a Hamiltonian function $H\colon S^1\times S^2\times T^\ast T^1\to\mathbb{R}$ with compact support such that $(C_\epsilon\times T^1)\cap \phi_{H}^1(C_\epsilon\times T^1)=\emptyset$ and $|\frac{\partial H}{\partial p}(x,y,z,p,q)|<1$ for any point $(x,y,z,p,q)$ in $S^2\times T^\ast T^1$.

Define the projection $\operatorname{pr}_4\colon S^2\times T^\ast T^1\to\mathbb{R}$ by $\operatorname{pr}_4(x,y,z,p,q)=p$ and
put $E=\sup\{|r|;r\in\operatorname{pr}_4(\cup_{t\in[0,1]}\operatorname{Supp}(H_t))\}$.
Let $R=(R_1,\ldots,R_n)$ be an element of $(\mathbb{R}_{>0})^n$ with $R_k>E$.
For $i\neq k$, we take a function $\rho_i \colon(-R_i,R_i)\to[0,1]$ with compact support such that $\rho_i=1$ in a neighborhood of $\{0\}$.
Let $\hat{H}\colon S^1\times S^2\times I_R^n\times T^n\to\mathbb{R}$ be a Hamiltonian function defined by
\[\hat{H}(t,x,y,z,p,q)=\prod_{i\neq k}\rho_i(p_i)\cdot H(t,x,y,z,p_k,q_k).\]
Since $R_k>E$, $\hat{H}$ has compact support in $S^1\times S^2\times I_R^n\times T^n$.
Then
\[\operatorname{pr}_\ast((X_{\hat{H}}^t)_{(x,y,z,p,q)})=\prod_{i\neq k}\rho_i(p_i)\cdot\frac{\partial H}{\partial p}(t,x,y,z,p_k,q_k),\]
where $\operatorname{pr}\colon S^2\times I_R^n\times T^n\to T^1$ is the projection defined by $\operatorname{pr}(x,y,z,p,q)=q_k$.
Since  $\rho_i= 1$ in a neighborhood of $\{0\}$, $(X\times T^n)\cap\phi^1_{\hat{H}}(X\times T^n)=\emptyset$. 
Since  $|\frac{\partial H}{\partial p}(t,x,y,z,p,q)|<1$ for any point $(t,x,y,z,p,q)$ in $S^1\times S^2\times T^\ast T^1$ and the image of $\rho_i$ is in $[0,1]$, $\mathcal{P}(\hat{H};(0_M,e))=\emptyset$.
Thus Proposition \ref{BPS displaceability} implies 
\[C(M,X,R;e)=C_{BPS}(M\times I_R^n\times T^n, X\times T^n;(0_M,e))=\infty\]
for any $e$ with $e\neq 0$.
\end{proof}

\section{Polterovich's invariant measure and Proof of Theorem \ref{measure capacity ineq}}\label{measure observation}

First, we review several definitions in order to fix the terminology.

\begin{defi}
Let $N$ be a manifold and $X$ a smooth vector field on $N$ generating a flow $\phi^t$.
For an invariant Borel measure $\mu$ of $\phi^t$ with compact support, its $\textit{rotation vector}$ $\rho(\mu,X)$ is an element of 1-dimensional homology $H_1(N;\mathbb{R})$ defined by
\[\langle \mathbf{l}^\ast,\rho(\mu,X)\rangle=\int_N\lambda(X)\mu,\]
for any cohomology class $\mathbf{l}^\ast$ of $H^1(N;\mathbb{R})$, where $\lambda$ is a closed 1-form representing $\mathbf{l}^\ast$.
\end{defi}
We can easily verify that $\int_N\lambda(X)\mu$ does not depend on the choice of $\lambda$.

We define relative symplectic capacities $C_{BPS}^P$ and $C^P$.
For a manifold $N$ and a vector field $X$ on $N$ generating a flow $\phi^t$, let $\mathfrak{M}(N,X)$ denote the set of invariant Borel measures of $\phi^t$ with compact support.
\begin{defi}
Let $Y$ be a compact subset of an open symplectic manifold $(N,\omega)$ and $\alpha\in[S^1,N]$ a free homotopy class in $N$.
For a cohomology class $\mathbf{l}^\ast\in H^1(N;\mathbb{R})$, we define the relative symplectic capacity $C_{BPS}^P(N,Y;\mathbf{l}^\ast,\alpha)$ by
\begin{align*}
&C_{BPS}^P(N,Y;\mathbf{l}^\ast,\alpha)\\
&=\inf \{K>0 ; \forall H\in C^\infty(N)\text{ such that }H|_Y\geq K,\\
&\exists\mu\in\mathfrak{M}(N,X_H)\text{ such that }|\langle \mathbf{l}^\ast,\rho(\mu,X_H) \rangle|\geq\mathbf{l}^\ast(\alpha)\}.\\
\end{align*}
We define the relative symplectic capacity $C_{BPS}^P(N,Y;\alpha)$ by 
\[C_{BPS}^P(N,Y;\alpha)=\sup_{\mathbf{l}^\ast\in H^1(N;\mathbb{R})}C_{BPS}^P(N,Y;\mathbf{l}^\ast,\alpha).\]
Let $X$ be a compact subset of a closed symplectic manifold $(M,\omega)$.
For an element $e=(e_1,\ldots,e_n)$ of $\mathbb{Z}^n$ and an element $R=(R_1,\ldots,R_n)$ of $(\mathbb{R}_{>0})^n$, we define the relative symplectic capacity $C^P(M,X,R;e)$ by
\[C^P(M,X,R;e)=C_{BPS}^P(M\times I_R^n\times T^n, X\times T^n;(0_M,e)).\]
\end{defi}
Note that for any positive real number $s$, $C_{BPS}^P(N,Y;s\mathbf{l}^\ast,\alpha)=C_{BPS}^P(N,Y;\mathbf{l}^\ast,\alpha)$.
Since every 1-periodic orbit representing a non-trivial homology class $\mathbf{a}$ determines an invariant measure with the rotation vector $\mathbf{a}$, we see that $C_{BPS}^P(N,Y,\alpha)\leq C_{BPS}(N,Y;\alpha)$ and $C^P(M,X,R;e)\leq C(M,X,R;e)$.

A diffeomorphism $\psi$ of $M$ is said to be \textit{a symplectomorphism} if  $\psi^\ast\omega=\omega$ and an isotopy $\{\psi^t\}_{t\in[0,1]}$ of $M$ is said to be \textit{a symplectic isotopy} if $\psi^0=\mathrm{id}$ and $(\psi^t)^\ast\omega=\omega$ for any $t$.
Let $\mathrm{Symp}(M,\omega)$ denote the group of symplectomorphisms of $(M,\omega)$ with compact support.
Let $\widetilde{\mathrm{Symp}}_0(M,\omega)$ denote the universal covering of the identity component of $\mathrm{Symp}(M,\omega)$.
An element of it is a homotopy class of symplectic isotopy $\{\psi^t\}_{t\in[0,1]}$ of $M$ relative to the end points $\psi^0=\mathrm{id}$ and $\psi^1=\psi$.

\begin{defi}
The \textit{flux homomorphism} $\operatorname{Flux}\colon\widetilde{\mathrm{Symp}}_0(M,\omega)\to H^1(M;\mathbb{R})$ is defined by
\[\operatorname{Flux}([\{\psi^t\}_{t\in[0,1]}])=\int_0^1\iota_{X^t}\omega dt,\]
for any element $[\{\psi^t\}_{t\in[0,1]}]$ of $\widetilde{\mathrm{Symp}}_0(M,\omega)$,
where $X^t$ is the (time-dependent) vector field which generates $\{\psi^t\}_t$.
\end{defi}
The flux homomorphism is known to be surjective.

To prove Theorem \ref{measure capacity ineq}, we explain Polterovich's result in \cite{P14}.

\begin{thm}[\cite{P14}]\label{Polterovich invariant measure}
Let $Y_1$ and $Y_2$ be non-displaceable compact subsets of a closed symplectic manifold $(N,\omega)$.
Assume that $Y_1\cap Y_2=\emptyset$ and there exists a symplectic isotopy $\{\psi^t\}_{t\in[0,1]}$ such that $\psi^1(Y_1)=Y_2$.
Put $\mathbf{l}^\ast=\operatorname{Flux}(\{\psi^t\})$.
Then for any positive real number $p$ and any Hamiltonian function $F\colon N\to\mathbb{R}$ such that $F|_{Y_1}\leq 0$ and $F|_{Y_2}\geq p$,
$\{\phi_F^t\}$ possesses an invariant measure $\mu$ such that $\operatorname{Supp}(\mu)\subset\operatorname{Supp}F$ and
\[|\langle \mathbf{l}^\ast,\rho(\mu,X_F) \rangle|\geq p.\]
\end{thm}

Let $\operatorname{pr}_1\colon M\times I_R^n\times T^n\to M$ denote the projections defined by $\operatorname{pr}_1(y,p,q)=y$.
Define the subset $S_R$ of $\mathbb{R}^3$ by $S_R=\partial\bar{I}_R^n$, more precisely,
\begin{align*}
S_R&=(\{-R_1,R_1\}\times[-R_2,R_2]\times\cdots\times[-R_n,R_n])\\
&\cup([-R_1,R_1]\times\{-R_2,R_2\}\times\cdots\times[-R_n,R_n])\\
&\cup\cdots\cup([-R_1,R_1]\times[-R_2,R_2]\times\cdots\times\{-R_n,R_n\}).\\
\end{align*}

\begin{proof}[Proof of Theorem \ref{measure capacity ineq}]
First, we prove $C^P(M,X,R;e)\leq\sum_{i=1}^nR_i\cdot |e_i|$.
Fix a cohomology class \[\mathbf{l}^\ast=\operatorname{pr}_1^\ast\mathbf{b}^\ast+a_1[dq_1]+\cdots+a_n[dq_n]\neq0\in H^1(M\times I_R^n\times T^n;\mathbb{R}),\]
where $a_1,\ldots,a_n$ are real numbers and $\mathbf{b}^\ast\in H^1(M;\mathbb{R})$ is a cohomology class of $M$.

If $(a_1,\ldots,a_n)=0$, then $l^\ast((0_M,e))=0$ and hence, by the definition, $C_{BPS}^P(M\times I_R^n\times T^n,X\times T^n;\mathbf{l}^\ast,(0_M,e))=0$.
Thus we may assume $(a_1,\ldots,a_n)\neq0$.

To use Theorem \ref{Polterovich invariant measure}, we prepare the symplectic isotopy $\{\psi^t\}_{t\in[0,1]}$.
Since $(a_1,\ldots,a_n)\neq0$, there exists a unique positive real number $K$ such that $(Ka_1,\ldots,Ka_n)\in S_R$.
Then we regard $I_R^n\times T^n$ as a subset of $T_{2R}^n\times T^n$.
Fix a point $x_0$ in $M$.
Since the flux homomorphism is surjective, there exists a symplectic isotopy $\{\psi_0^t\}_{t\in[0,1]}$ of $M$
such that $\operatorname{Flux}(\{\psi_0^t\}_{t\in[0,1]})=K\mathbf{b}^\ast$.
Let $\{\psi^t\}$ be the symplectic isotopy of $M\times T_{2R}^n\times T^n$  defined by 
\[\psi^t(x,p_1,\ldots,p_n,q_1,\ldots,q_n)=(\psi_0^t(x),p_1+Ka_1t,\ldots,p_n+Ka_nt,q_1,\ldots,q_n).\]
Then \[\operatorname{Flux}(\{\psi^t\}_{t\in[0,1]})=Ka_1[dq_1]+\cdots+Ka_n[dq_n]+K\operatorname{pr}_1^\ast\mathbf{b}^\ast=K\mathbf{l}^\ast.\]
Assume that a Hamiltonian function $H\colon  M\times I_R^n\times T^n\to\mathbb{R}$ satisfies $H|_{X\times T^n}\geq\sum_{i=1}^nR_i\cdot|e_i|$.
We regard $H$ as a Hamiltonian function on $M\times T_{2R}^n\times T^n$.
Since $\psi^1((\psi_0^1)^{-1}(X)\times\{(-Ka_1,\ldots,-Ka_n)\}\times T^n)=X\times\{0\}\times T^n$ and $H|_{(\psi_0^1)^{-1}(X)\times \{(-Ka_1,\ldots,-Ka_n)\}\times T^n}=0$, Theorem \ref{Polterovich invariant measure} implies that there exists an invariant measure $\mu$ on $M\times I_R^n\times T^n$ such that 
\[|\langle R_1[dq_1]+\cdots+R_n[dq_n],\rho(\mu,X_H) \rangle|\geq\sum_{i=1}^nR_i\cdot|e_i|.\]
Since $(Ka_1,\ldots,Ka_n)\in S_R$, $K\mathbf{l}^\ast((0_M.e))\leq\sum_{i=1}^nR_i\cdot|e_i|$.
Thus, for any cohomology class $\mathbf{l}^\ast$ with $(a_1,\ldots,a_n)\neq0$, 
\begin{align*}
&C_{BPS}^P(M\times I_R^n\times T^n,X\times T^n;\mathbf{l}^\ast,(0_M,e))\\
&=C_{BPS}^P(M\times I_R^n\times T^n,X\times T^n;K\mathbf{l}^\ast,(0_M,e))\leq\sum_{i=1}^nR_i\cdot|e_i|.\\
\end{align*}
Since $C_{BPS}^P(M\times I_R^n\times T^n,X\times T^n;0,(0_M,e))=0$, $C^P(M,X,R;e)\leq\sum_{i=1}^nR_i\cdot|e_i|$.

Now, we prove $C^P(M,X,R;e)\geq\sum_{i=1}^nR_i\cdot |e_i|$.
Without loss of generality, we can assume that every $e_i$ is non-negative.
Let $\epsilon$ be a positive real number.

Let $\lambda=l_1dq_1+\cdots+l_ndq_n$ be a closed 1-form on $M\times I_R^n\times T^n$ such that every $l_i$ is non-negative.
Then, we can take a function $\rho\colon (-R_1,R_1)\times\cdots\times(-R_n,R_n)\to\mathbb{R}$ with compact support and satisfying the following conditions.
\begin{itemize}
\item $\rho= \max\{\sum_iR_i\cdot |e_i|-\epsilon,0\}$ in a neighborhood of \{0\},

\item $\rho(p)\leq \max\{\sum_iR_i\cdot |e_i|-\epsilon,0\}$ for any $p\in (-R_1,R_1)\times\cdots\times(-R_n,R_n)$,
\item $|\frac{\partial\rho}{\partial p_i}(p)|<|e_i| $ for any $p\in (-R_1,R_1)\times\cdots\times(-R_n,R_n)$.
\end{itemize}

Let $H\colon M\times I_R^n\times T^n\to\mathbb{R}$ be a Hamiltonian function defined by $H(x,p,q)=\rho(p)$.
Since every $e_i$ and $l_i$ is non-negative,  for any point $(x,p,q)$ of $M\times I_R^n\times T^n$,
\[|\lambda((X_H)_{x,p,q})|=|\sum_il_i\cdot\frac{\partial\rho}{\partial p_i}(p)|\leq\sum_i|l_i|\cdot |e_i|=\lambda((0_M,e)).\]
Hence for any Borel measure $\mu$, $\int\lambda(X_H)\mu\leq\lambda((0_M,e))$.
Thus for any $\epsilon$,
\[C_{BPS}^P(M\times I_R^n\times T^n, X\times T^n;[\lambda],(0_M,e))\geq\sum_iR_i\cdot |e_i|-\epsilon.\]
Therefore, by the definition of $C^P$, $C^P(M,X,R;e)\geq\sum_iR_i\cdot |e_i|$.
\end{proof}

\end{document}